\documentclass[conference]{IEEEtran}
\IEEEoverridecommandlockouts
\usepackage{cite}
\usepackage{amsmath,amssymb,amsfonts}
\usepackage{comment}
\usepackage{algorithmic}
\usepackage{graphicx}
\usepackage{textcomp}
\usepackage{xcolor}
\usepackage{url}
\usepackage{hyperref}
\usepackage{amsthm}

\newtheorem{theorem}{Theorem}

\newtheorem{remark}{Remark}

\def\BibTeX{{\rm B\kern-.05em{\sc i\kern-.025em b}\kern-.08em
    T\kern-.1667em\lower.7ex\hbox{E}\kern-.125emX}}
\begin{document}

\newcommand{\SWFnote}[1]{\textcolor{red}{SWF: #1}}

\title{
% Mean-field Control Barrier Functions for Swarm Tracking and Avoidance*\\
Mean-Field Control Barrier Functions:
\\
A Framework for Real-Time Swarm Control
\thanks{This work was partially funded by NSF DMS award 2110745.}
\thanks{The authors contributed equally.}
}

\author{\IEEEauthorblockN{Samy Wu Fung}
\IEEEauthorblockA{\textit{Department of Applied Mathematics and Statistics} \\
% \textit{Department of Computer Science}\\
\textit{Colorado School of Mines}\\
Golden, USA \\
swufung@mines.edu}
\and
\IEEEauthorblockN{Levon Nurbekyan}
\IEEEauthorblockA{\textit{Department of Mathematics} \\
\textit{Emory University}\\
Atlanta, USA \\
lnurbek@emory.edu}
}

\maketitle

\begin{abstract}
Control Barrier Functions (CBFs) are an effective methodology to ensure safety and performative efficacy in \emph{real-time} control applications such as power systems, resource allocation, autonomous vehicles, robotics, etc. 
This approach ensures safety independently of the high-level tasks that may have been pre-planned off-line. 
For example, CBFs can be used to guarantee that a vehicle will remain in its lane.
However, when the number of agents is large, computation of CBFs can suffer from the curse of dimensionality in the multi-agent setting.
In this work, we present Mean-field Control Barrier Functions (MF-CBFs), which extends the CBF framework to the mean-field (or swarm control) setting.
The core idea is to model a population of agents as probability measures in the state space and build corresponding control barrier functions. Similar to traditional CBFs, we derive safety constraints on the (distributed) controls but now relying on the differential calculus in the space of probability measures.
\end{abstract}

\begin{IEEEkeywords}
real-time control, safety, optimal control, barrier functions, mean-field, swarm control, robotics
\end{IEEEkeywords}

\section{Introduction}

\begin{comment}
\begin{itemize}
    \item what is control
    \item importance of real-time control
    \item brief descriptions of barrier functions (citations)
    \item challenges arising from numerous agents
    \item our contribution with a brief primer on mean field theory
\end{itemize}    
\end{comment}

Control problems are ubiquitous in applications, including aerospace engineering, robotics, economics, finance, power systems management, etc. Typically, one has a system that can be manipulated by applying a control, and the goal is to drive the system to certain states, maintaining suitable constraints and acting as economically as possible.

When the constraints are known ahead of time, one can solve for controls that maintain these constraints offline. However, there are numerous situations where some constraints are unknown before deployment and must be dealt with online. 
Examples of such constraints include avoiding an unexpected obstacle or maintaining a certain distance from an agent with unknown dynamics, e.g., pedestrians.

For safety-critical applications, real-time computation of effective constraint-maintaining controllers is crucial. Control Barrier Functions is an effective framework for computing such controllers~\cite{ames2016control,ames2019control}. In short, one represents the state constraints as a sublevel set of a suitable function, which then yields the set of safe (constraint-maintaining) controls via differential inequality. Hence, one can replace the nominal control by the closest possible safe control.

The CBF methodology is appealing due to its theoretical guarantees, local nature, and computational benefits. 
Indeed, the set of safe controls at a given state depends only on the data at the current state. 
Additionally, the set of safe controls is convex for control-affine systems, and state-of-the-art convex optimization algorithms are applicable for fast computations of safe controllers. 
Finally, no computation is necessary when the nominal control is safe.

In this paper, we extend the CBF methodology to infinite-dimensional control problems in the space of probability measures. 
Such control problems are often called \textit{mean-field control problems} as one aims to control distributions of states rather than a single state~\cite{lasry2006jeux, lasry2007mean, fornasier2014mean, gomes2014mean, ruthotto2020machine,lin2021alternating}. Hence, we call the framework \textit{Mean-field Control Barrier Functions (MF-CBFs)}.

The mean-field framework is an efficient way of modeling multi-agent systems~\cite{lasry2006jeux, lasry2007mean, fornasier2014mean, gomes2014mean, ruthotto2020machine,lin2021alternating}. Indeed, the dynamics of a swarm in a state space are equivalent to the dynamics of the empirical distribution of the swarm in the space of probability measures. Modeling the swarm behavior via mean-field framework has several benefits. First, the mathematical analysis is performed in the space of probability measures, which is \emph{independent of the swarm size} unlike the product space of the joint state of the swarm. Second, instead of searching for individual controls, one can search for a common (distributed) control in a feedback form, significantly reducing the problem dimension. 
See~\cite{onken2020neural, onken2022neural,bansal2020provably,bansal2021deepreach,chen2020guaranteed,zhang2024gcbf} for the challenges occurring in high-dimensional multi-agent control problems.

Our main contributions in this paper are as follows.
\begin{itemize}
    \item Formulation of the CBF framework for mean-field control problems in the space of probability measures.
    \item Derivation of MF-CBFs suitable for swarm avoidance and tracking.
    \item Numerical experiments of swarm avoidance and tracking with up to 200 agents.
\end{itemize}

This paper is organized as follows. In Section~\ref{sec: background}, we review preliminary concepts on control barrier functions and their challenges in the mult-agent setting. In Section~\ref{sec: MF-CBF}, we present the mean-field control barrier function framework. In Section~\ref{sec: experiments}, we walk through some illustrative examples of swarm tracking and avoidance with up 200 agents.  

\section{Background: Control Barrier Functions}
\label{sec: background}

In this section, we provide a brief introduction to CBFs and refer to~\cite{ames2019control,ames2016control} for a more in-depth discussion of CBFs.

\subsection{Control Problems}

We consider deterministic finite time-horizon control problems, where a system obeys the dynamics
\begin{equation}
    \partial_s z(s) = f(s,z(s),u(s)), \quad z(t)= x, \quad t \leq  s \leq  T.
    \label{eq: dynamics_single_agent}
\end{equation}
Above, $x \in \mathbb{R}^d$ is the initial state, $T$ is the time-horizon, $t<T$ is an initial time, and $z:[t,T] \to \mathbb{R}^d$ and $u:[t,T] \to U \subset \mathbb{R}^m$ are, correspondingly, the state and the control of the system as functions of time. Furthermore, the function $f \colon [t,T]\times \mathbb{R}^d\times U \to \mathbb{R}^d$ models the evolution of the state $z$ in response to the control $u \colon [t,T]\to U$. We say that the system is \textit{control-affine} if the dynamics are affine with respect to the control; that is,
\begin{equation}
    f(s,z,u) = A(s,z) + B(s,z)u
    \label{eq: affine_controller}
\end{equation}
where $A,B$ are possibly nonlinear maps of time and state. Control-affine systems are ubiquitous and cover a wide range of applications~\cite{kunisch2020semiglobal,flemingsoner06,carrillo2013modeling,ames2019control}. 

In control problems, one searches for controls for achieving a suitable goal. For instance, one might seek controls for reaching a final destination while avoiding dangerous zones; that is,
\begin{equation}\label{eq:reach}
    \text{find}~u(\cdot)~\text{s.t.}~z(T)\in R \quad \text{and} \quad z(s)\notin D,~\forall t\leq s \leq T,
\end{equation}
where $R$ is the destination set, and $D$ is the dangerous zone. See~\cite{bansal2017hamilton} for a detailed discussion on problems of type~\eqref{eq:reach}.

A large class of control problems seek to control a system in an optimal manner; that is, search for controls that minimize a \textit{cost functional}
\begin{equation}\label{eq:OC}
	\begin{split}
    &\inf_{u(s)\in U}  \int_{t}^T L\big(s,z(s), u(s)\big) \, d s +  G\big(z(T)\big) \\ \text{s.t. }~&\eqref{eq: dynamics_single_agent}~\text{holds},
    \end{split}
\end{equation}
where $L \colon [0,T]\times \mathbb{R}^d \times U \to \mathbb{R}$ is the \textit{running cost} (or the \textit{Lagrangian}), $G \colon \mathbb{R}^d \to \mathbb{R}$ is the \textit{terminal cost}, and $\Phi$ is the so-called \textit{value function} or \textit{optimal cost-to-go}.

Problems such as~\eqref{eq:OC} are called \textit{optimal control} problems. A solution $u^*_{t,x}$ of \eqref{eq:OC} is called an \textit{optimal control}. Accordingly, the $z^*_{t,x}$ which corresponds to $u^*_{t,x}$ is called an \textit{optimal trajectory}. See, for instance,~\cite{flemingsoner06} for a detailed account on optimal control problems.

Whether it is the reachability problem~\eqref{eq:reach} or the optimal control problem~\eqref{eq:OC}, it is advantageous to find controls in a \textit{feedback form}; that is,
\begin{equation}\label{eq:oc_feedbackform}
    u^*_{t,x}(s)=q(s,z^*_{t,x}(s)),\quad t\leq s \leq T.
\end{equation}
Here, the function $q$ is called a \textit{policy function}. Hence, instead of searching for controls separately at each initial point $(t,x)$ one can search for a suitable policy function that yields the desired controls for all initial points at once.

\begin{comment}
\begin{equation}\label{eq:Joc}
     \int_{t}^T L\big(s,z(s), u(s)\big) d s  \,\, +  \,\, G\big(z(T)\big),
\end{equation}

Suitable regularity requirements on $f,L,G,U$ are necessary (see \cite[Sec. I.3, I.8-9]{flemingsoner06} for more details) for the . The goal of the optimal control problem is to find a control $u^*_{t,x}$ that incurs the minimal cost, i.e,

Another central object in the optimal control theory is the \textit{Hamilton-Jacobi-Bellman} partial differential equation (PDE). One first defines the Hamiltonian of the system by 
\begin{equation}
    \begin{split}
        H(s,z,p) = \sup_{u \in U} \left\{ -p \cdot f(s,z,u)  - L(s, z, u) \right\},
    \end{split}
\end{equation}
where $p \in \mathbb{R}^d$ is the \emph{adjoint state}. Again, under suitable regularity assumptions, one can prove that the value function is the \textit{viscosity solution} of the HJB PDE
\begin{equation}\label{eq:HJB}
    \begin{cases}
    -\partial_t \Phi(s,z)+H(s,z,\nabla \Phi(s,z))=0,\\
    \Phi(T,z)=G(z).
    \end{cases}
\end{equation}    
\end{comment}

\subsection{Control Barrier Functions}

Controls in feedback-form~\eqref{eq:oc_feedbackform} are satisfactory when we have access to problem data, such as $f,L,G$ in~\eqref{eq:OC} or $R,D$ in~\eqref{eq:reach}, that encode the essential features of the problem. However, $q$ is not designed to handle \emph{unforeseen} circumstances such as \textit{real-time} collision and danger zone avoidance, or tracking.

To this end, one can enhance the nominal (pre-computed) controller $u^*$ with mission-oriented filters that use sensor data to adjust $u^*$ in real-time when, e.g., an unforeseen obstacle appears.

A successful approach to filter $u^*$ are CBFs. The basic idea underlying CBFs is as follows. Consider the dynamics of agents in~\eqref{eq: dynamics_single_agent}, where $u$ is some control. Furthermore, assume that $h$ encodes safety constraints or other goals so that it is desireable to have
\begin{equation}\label{eq:h>=0}
h(z(s)) \geq 0, \quad \forall s\geq t.     
\end{equation}
One way to ensure this is to impose
\begin{equation}\label{eq:cbf_condition}
\frac{d}{ds}h(z(s))\geq -\alpha(h(z(s))),\quad \forall s\geq t,
\end{equation}
where $\alpha:\mathbb{R} \to \mathbb{R}$ is a strictly increasing smooth function such that $\alpha(0)=0$~\cite{ames2019control,ames2016control}. Feeding~\eqref{eq:cbf_condition} in the dynamics of $z$, we obtain
\begin{equation}
\nabla h(z(s))\cdot f(s,z(s),u(s)) \geq -\alpha(h(z(s))),\quad \forall s\geq t.
\end{equation}

Thus, one can adjust $u^*$ in real time by solving the following quadratic program
\begin{equation}
    \begin{split}
    &u_{\text{CBF}}(s) \in \underset{{u \in U}}{\operatorname{argmin}} \;\| u - u^*(s) \|^2
    \\
    &\mbox{s.t. } \nabla h(z(s))\cdot f(s,z(s),u) \geq -\alpha(h(z(s))).
    \end{split}
    \label{eq: general_CBF}
\end{equation}

The construction of $h$ depends on the application and on the type of live sensors. 

\subsubsection{Existing Challenges}
One of the main drawbacks of CBFs is that it is prone to the \textit{curse of dimensionality}~\cite{bellman57} for multi-agent systems. The latter appear, for instance, in applications such as the Glider Coordinated Control System for ocean monitoring~\cite{paley2008cooperative} and informative Unmanned Aerial Vehicle (UAV) mission planning~\cite{glock2020mission}.
Indeed, assume that $z_1,z_2,\cdots,z_n$ represent the states of $n$ control systems (agents) obeying the dynamics
\begin{equation}
    \begin{split}
    &\partial_s z_i(s) = f_i(z_i(s),u_i(s)), \quad 1\leq i \leq n,
    \\
    &z(t)= x, \quad t \leq  s \leq  T.
    \end{split}
    \label{eq: dynamics_multiple_agent}
\end{equation}
Furthermore, let $h_1,h_2,\cdots,h_n$ encode the safety requirements for individual agents $1,2,\cdots,n$, respectively. Finally, let $h_{ij}$ encode a mutual safety requirement for a pair of agents $1\leq i \neq j \leq n$. For example, the functions
\begin{equation*}
h_{ij}(z_i,z_j)=\|z_i-z_j\|^2-\epsilon_{\text{safe}}^2    
\end{equation*}
reflect the requirement that the distance between two agents should be at least $\epsilon_{\text{safe}} > 0$.

A common approach to study such multi-agent systems is to concatenate all states and controls into one ``large'' agent. Specifically, let 
\begin{equation*}
\begin{split}
    \mathbf{z}=&(z_1,z_2,\ldots,z_n)\in \mathbb{R}^{d\times n},
    \\
    \mathbf{u}=&(u_1,u_2,\ldots,u_n)\in U^n \subset \mathbb{R}^{m \times n},\\
    ~\mathbf{f}(\mathbf{z},\mathbf{u})=&(f_1(z_1,u_1),f_2(z_2,u_2),\ldots,f_n(z_n,u_n))\in \mathbb{R}^{d\times n},
\end{split}
\end{equation*}
and
\begin{equation*}
    \begin{split}
        &\mathbf{h}_i(\mathbf{z})=h_i(z_i),~\forall 1\leq i \leq n, 
        \\
        &\mathbf{h}_{ij}(\mathbf{z})=h_{ij}(z_i,z_j),~\forall 1\leq i \neq j \leq n.
    \end{split}
\end{equation*}
Then the quadratic program for computing safe controls is
\begin{equation}
    \begin{split}
    &\mathbf{u}_{\text{CBF}}(s) \in \underset{\mathbf{u} \in U^n}{\operatorname{argmin}} \;\| \mathbf{u} - \mathbf{u}^*(s) \|^2
    \\
    \mbox{s.t.}~& \nabla \mathbf{h}_i(\mathbf{z}(s))\cdot \mathbf{f}(\mathbf{z}(s),\mathbf{u}) \geq -\alpha(\mathbf{h}_i(\mathbf{z}(s))),~\forall i,\\
    & \nabla \mathbf{h}_{ij}(\mathbf{z}(s))\cdot \mathbf{f}(\mathbf{z}(s),\mathbf{u}) \geq -\alpha(\mathbf{h}_{ij}(\mathbf{z}(s))),~\forall i\neq j.
    \end{split}
    \label{eq: general_CBF_swarm}
\end{equation}
where $\mathbf{u}^* = (u_1^*, u_2^*, \ldots, u_n^*) \in U^n$ is the concatenated nominal controllers. 

A few important remarks are in order. 
\begin{enumerate}
    \item The dimension of the quadratic problem~\eqref{eq: general_CBF_swarm} is $n \times m$ as opposed to $m$ in~\eqref{eq: general_CBF}.
    \item The number of inequality constraints (excluding the a priori constraints, $\mathbf{u} \in U^n$, on the controls) in~\eqref{eq: general_CBF_swarm} is $n^2$ as opposed to $1$ in~\eqref{eq: general_CBF}. Hence, for control-affine systems, going from a single agent to $n$ agents yields going from one additional half-space constraint to $n^2$ additional half-space constraints.
    \item If individual agents need more than one CBF for their safety requirements, and there are additional collective safety or goal requirements beyond pair-to-pair interactions, the total number of CBF or inequality constraints in the projection problem will be even larger. 
    \item The inequality constraints in~\eqref{eq: general_CBF_swarm} are \textit{coupled} and in general cannot be solved on the level of individual agents. Smart decoupling techniques mitigate the coupling issue, but still lead to $n$ quadratic programs with $O(n)$ inequality constraints, which are challenging to solve online for $n\gg 10$~\cite{chen2020guaranteed}.
\end{enumerate}

\section{Mean-field Control-barrier Functions}
\label{sec: MF-CBF}

To mitigate the challenges of computing CBFs for large multi-agent swarms, we introduce \emph{Mean-Field Control Barrier Functions} (MF-CBFs). The core idea is to formulate the swarm CBF problem in~\eqref{eq: general_CBF_swarm} in the space of \emph{distributions}. This allows us to, e.g., represent the inter-agent distance requirements in~\eqref{eq: general_CBF_swarm}
by a lower bound on a \emph{single} mean-field function.

\subsection{Mean-field Control}

Assume that we have a population (swarm) of agents in the state space, where an individual agent follows the dynamics~\eqref{eq: dynamics_single_agent}. Furthermore, assume that the distribution of the population in the state space at time $s$ is described by the probability measure $\rho(s,\cdot)$, where we often use the same notation for a measure and its density function.

In the mean-field control setting, we consider only feedback-form (distributed) controls and assume that all agents adopt the same policy function. Hence, given a policy function $q$ adopted by the population, the density evolves according to the continuity equation
\begin{equation}\label{eq:cont_eq}
\begin{cases}
    \partial_s \rho(s,x) + \nabla \cdot (\rho(s,x)f(s,x,q(s,x)) = 0,~s\in (0,T),\\
    \rho(0,x)=\rho_0(x),
\end{cases}
\end{equation}
where $\rho_0$ is the initial distribution of the population, and $\nabla \cdot$ is the divergence operator with respect to the state variable $x$.

The mean-field control or swarm-control problem is then formulated as
\begin{equation}
\begin{split}
    &\inf_{q(s,x) \in U} \int_0^T \mathcal{L}(s,\rho(s,\cdot),q(s,\cdot)) ds+\mathcal{G}(\rho(T,\cdot))\\
    \text{s.t. }~&\eqref{eq:cont_eq}~\text{holds},
\end{split}
\end{equation}
where $\mathcal{L}$ and $\mathcal{G}$ are mean-field running and terminal costs, respectively. The dependencies of $\mathcal{L},\mathcal{G}$ on $\rho$ encode the swarm behavior that one attempts to model.

\subsection{Mean-field control-barrier functions}

Analogous to single-agent control problems one may have safety constraints or goals for swarm control problems. Building on the mean-field control framework, we propose mean-field control-barrier functions (MF-CBFs) for efficiently handling safety constraints and other goals. 
% \begin{align}
%     C &= \{\rho \in \mathcal{P}(\mathbb{R}^d) \colon \mathcal{H}(\rho) \geq 0\}
%     \\
%     \partial C &= \{\rho \in \mathcal{P}(\mathbb{R}^d) \colon \mathcal{H}(\rho) = 0\}
%     \\
%     \text{int}(C) &=\{\rho \in \mathcal{P}(\mathbb{R}^d) \colon \mathcal{H}(\rho) > 0\}
% \end{align}

% \textcolor{red}{include time component?}

% \begin{definition}
%     Let $C$ be the superlevel set of a continuously differentiable function $\mathcal{H} \colon \mathcal{P}(\mathbb{R}^d) \to \mathbb{R}$, then $\mathcal{H}$ is a mean-field control barrier function (MF-CBF) if there exists a strictly increasing function $\alpha \colon \mathbb{R} \to \mathbb{R}$ with $\alpha(0) = 0$ such that 
%     \begin{equation}\label{eq: mf_cbf}
%         \sup_{u(x)\in U} \int_{\mathbb{R}^d} \nabla^W_{\rho} \mathcal{H} (\rho) \cdot f(x,u(x)) \rho(x) dx \geq -\alpha\left(\mathcal{H}(\rho)\right).
%     \end{equation}
%     for all $\rho \in \mathcal{P}(\mathbb{R}^d)$, where $\nabla^W_{\rho}$ is the Wasserstein derivative of the functional $\mathcal{H}$.
% \end{definition}

% In this case, the set of CBF or safe controls is denoted by
% \begin{equation}\label{eq:mf_K_cbf}
%     \begin{split}
%     &\mathcal{K}_{\text{CBF}}(\rho)
%     \\
%     &=\left\{ q~:~\int_{\mathbb{R}^d} \nabla^W_\rho \mathcal{H} (\rho) \cdot f(x,q(x)) \rho(x) dx \geq -\alpha\left(\mathcal{H}(\rho)\right) \right\}.
%     \end{split}
% \end{equation}
% with $q(x)\in U$.

Assume that $\mathcal{H}$ encodes, possibly time-dependent, safety constraints or other goals of the swarm; that is, one must have
\begin{equation}\label{eq:calH>=0}
\mathcal{H}(s,\rho(s,\cdot))\geq 0, \quad \forall s\geq 0.
\end{equation}
As in~\eqref{eq:h>=0}, one can ensure this previous inequality by imposing
\begin{equation}\label{eq:MFCBF_condition}
\frac{d}{ds}\mathcal{H}(s,\rho(s,\cdot))\geq -\alpha(\mathcal{H}(s,\rho(s,\cdot))),\quad \forall s\geq 0,
\end{equation}
where $\alpha:\mathbb{R}\to \mathbb{R}$ is again a strictly increasing smooth function such that $\alpha(0)=0$. We provide a simple proof of this statement for completeness.
\begin{theorem}
    Let $\alpha \in C(\mathbb{R})$ be a strictly increasing function such that $\alpha(0)=0$, and $\rho(s,\cdot)$ and $\mathcal{H}(s,\cdot)$ be such that $\omega(s)=\mathcal{H}(s,\rho(s,\cdot)),~s\geq 0$, is a continuously differentiable function. Then $\omega(0)\geq 0$ and~\eqref{eq:MFCBF_condition} guarantee that $\omega(s)\geq 0$ for all $s\geq 0$. 
\end{theorem}
\begin{proof}
    Assume by contradiction that
    \begin{equation*}
        N=\left\{s\geq 0~:~\omega(s)<0\right\}\neq \emptyset.
    \end{equation*}
    Since $\omega$ is continuous, we have that $N\subset (0,\infty)$ is an open set; hence, $N$ is a union of disjoint open intervals. Let $(a,b)\subset N$ be one such interval. Then we have that $0<a<\infty$, and
    \begin{equation}\label{eq:omega_sign}
        \omega(a)=0,\quad \omega(s)<0,\quad s\in (a,b).
    \end{equation}
    Hence, we have that
    \begin{equation*}
        \omega'(s) \geq -\alpha(\omega(s))>-\alpha(0)=0,\quad s\in (a,b).
    \end{equation*}
    Thus, $\omega$ is strictly increasing in $[a,b)$, which contradicts to~\eqref{eq:omega_sign}.
\end{proof}

Next, we find the constraints that~\eqref{eq:MFCBF_condition} imposes on the feedback (distributed) controls that the swarm should adopt. 
\begin{theorem}\label{thm:main_inclusion}
    Let
    \begin{equation}\label{eq:K_CBF_gen}
\begin{split}
    &\mathcal{K}_{\text{CBF}}(s,\rho)\\
    =&\bigg\{ q~:~\int_{\mathbb{R}^d} \nabla  \delta_\rho \mathcal{H}(s,\rho) \cdot f(s,x,q(x)) \rho(x)dx \\
    &\geq -\partial_s \mathcal{H}(s,\rho) -\alpha\left(\mathcal{H}(s,\rho)\right) \bigg\},\quad s\geq 0,
\end{split}
\end{equation}
and assume that the swarm adopts a policy function $q(s,\cdot)$ so that $\rho(\cdot,s)$ evolves according to~\eqref{eq:cont_eq}. Then~\eqref{eq:MFCBF_condition} is equivalent to
\begin{equation}\label{eq:MFCBF_inclusion}
    q(s,\cdot) \in \mathcal{K}_{\text{CBF}}(s,\rho(s,\cdot)),\quad \forall s\geq 0.
\end{equation}
\end{theorem}
\begin{remark}
    For simplicity, we assume necessary regularity for~\eqref{eq:cont_eq} to be well-posed and smooth enough for differential calculus. Additionally, we assume that $\rho(s,\cdot)$ decays fast enough at infinity, e.g., compactly supported, which can be ensured by fast decaying $\rho_0$ and smooth, e.g. Lipschitz, $q(s,\cdot)$ for $s\geq 0$~\cite{ambrosio08}.
\end{remark}
\begin{proof}
We have that
\begin{equation*}
    \begin{split}
        &\frac{d}{ds}\mathcal{H}(s,\rho(s,\cdot))\\
        =&\partial_s \mathcal{H}(s,\rho(s,\cdot))+\int_{\mathbb{R}^d} \delta_\rho \mathcal{H}(s,\rho(s,\cdot)) \partial_s \rho(s,x) dx,
    \end{split}
\end{equation*}
where $\delta_\rho \mathcal{H}$ is the Fr\'{e}chet derivative with respect to $\rho$ variable. Taking into account~\eqref{eq:cont_eq} and integrating by parts we find that
\begin{equation*}
    \begin{split}
        &\int_{\mathbb{R}^d} \delta_\rho \mathcal{H}(s,\rho(s,\cdot)) \partial_s \rho(s,x) dx\\
        =&-\int_{\mathbb{R}^d} \delta_\rho \mathcal{H}(s,\rho(s,\cdot)) \nabla \cdot (\rho(s,x)f(s,x,q(s,x))dx\\
        =&\int_{\mathbb{R}^d} \nabla  \delta_\rho \mathcal{H}(s,\rho(s,\cdot)) \cdot f(s,x,q(s,x)) \rho(s,x)dx.
    \end{split}
\end{equation*}
Hence,~\eqref{eq:MFCBF_condition} reduces to
\begin{equation}\label{eq:MFCBF_expanded}
    \begin{split}
        &\partial_s \mathcal{H}(s,\rho(s,\cdot))\\
        &+\int_{\mathbb{R}^d} \nabla  \delta_\rho \mathcal{H}(s,\rho(s,\cdot)) \cdot f(s,x,q(s,x)) \rho(s,x)dx\\
        \geq& -\alpha (\mathcal{H}(s,\rho(s,\cdot))),\quad \forall s\geq 0,
    \end{split}
\end{equation}    
or, equivalently,~\eqref{eq:MFCBF_inclusion}.
\end{proof}

Theorem~\ref{thm:main_inclusion} provides constraints on the policy function that ensure safe controls or mission accomplishing controls for the swarm. The mean-field analog of~\eqref{eq: general_CBF} is
\begin{equation}
    \begin{split}
    &q_{\text{CBF}}(s,\cdot) \in \underset{\substack{q \in \mathcal{K}_{\text{CBF}}(s,\rho(s,\cdot))\\q(x)\in U}}{\operatorname{argmin}} \; \| q - q^*(s,\cdot) \|^2_{L^2(\rho(s,\cdot))},
    \end{split}
    \label{eq: MF_CBF}
\end{equation}
where $q^*(s,\cdot),~s\geq 0$ is the nominal control of the swarm.

\begin{figure*}[tb]
    \centering
    \begin{tabular}{ccc}
    \multicolumn{3}{c}{Swarm Avoidance Example: 3D Double Integrator}
    \\
    \includegraphics[width=0.22\textwidth]{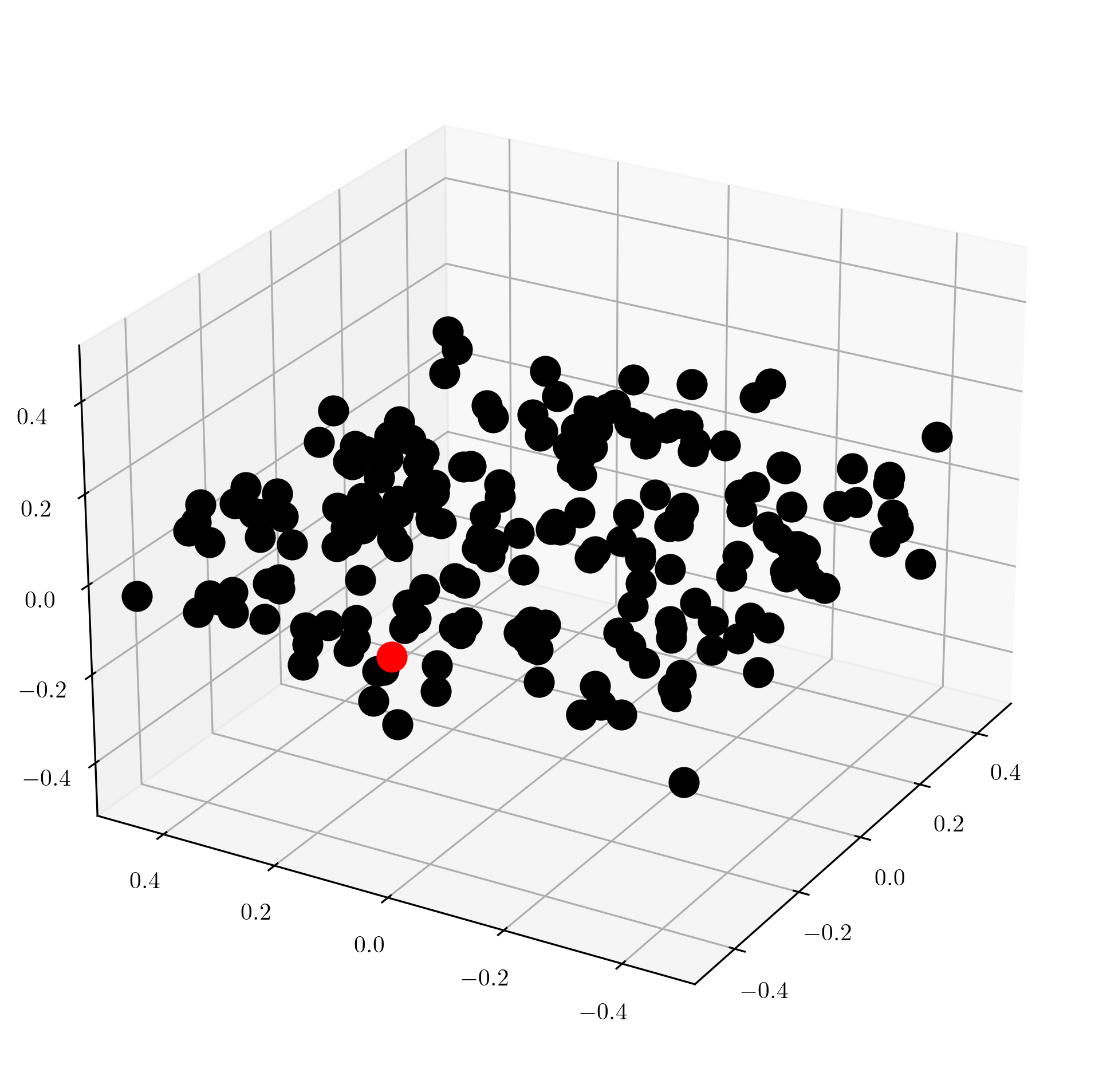}
    &
    \includegraphics[width=0.22\textwidth]{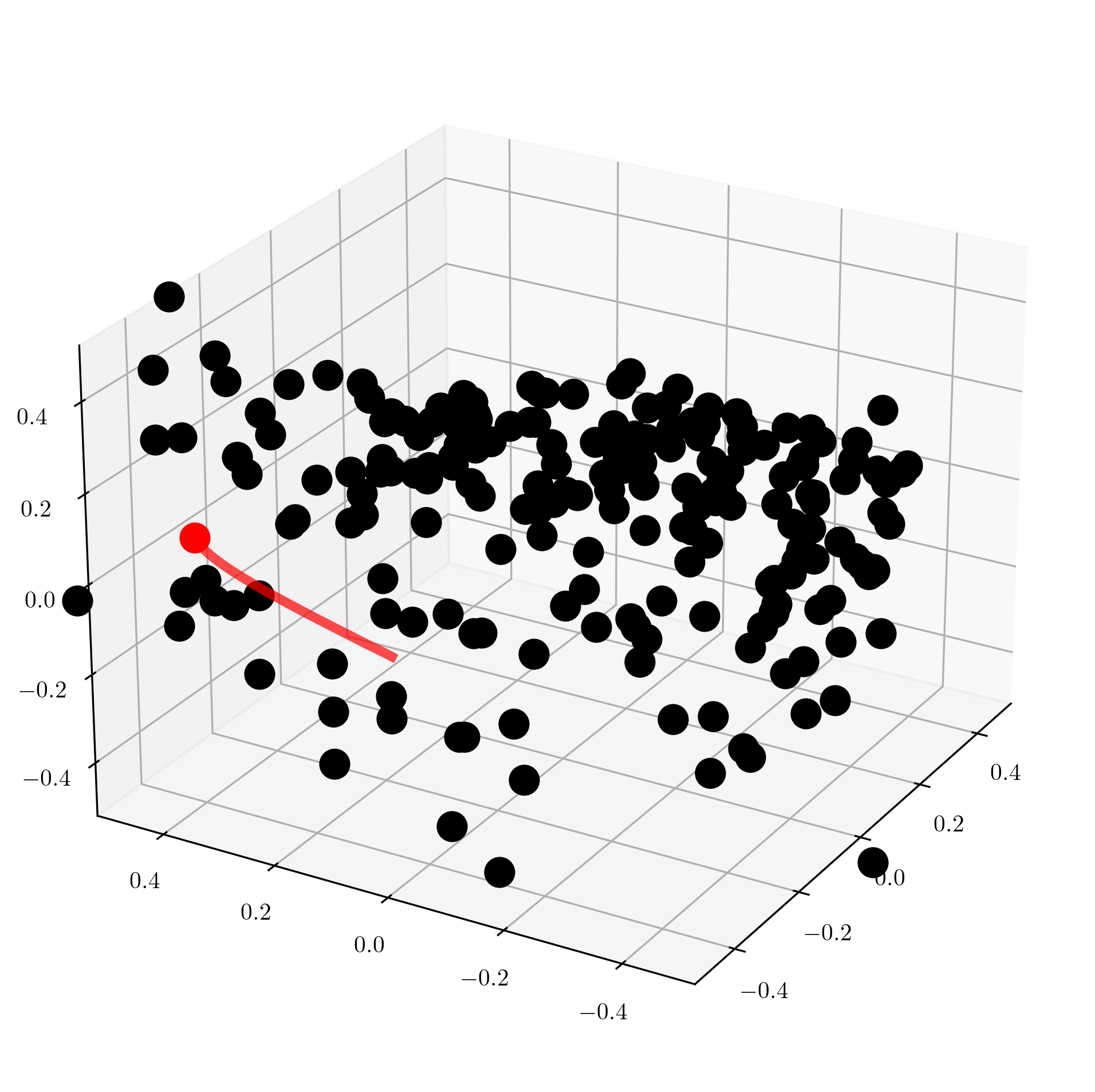}
    &
    \includegraphics[width=0.22\textwidth]{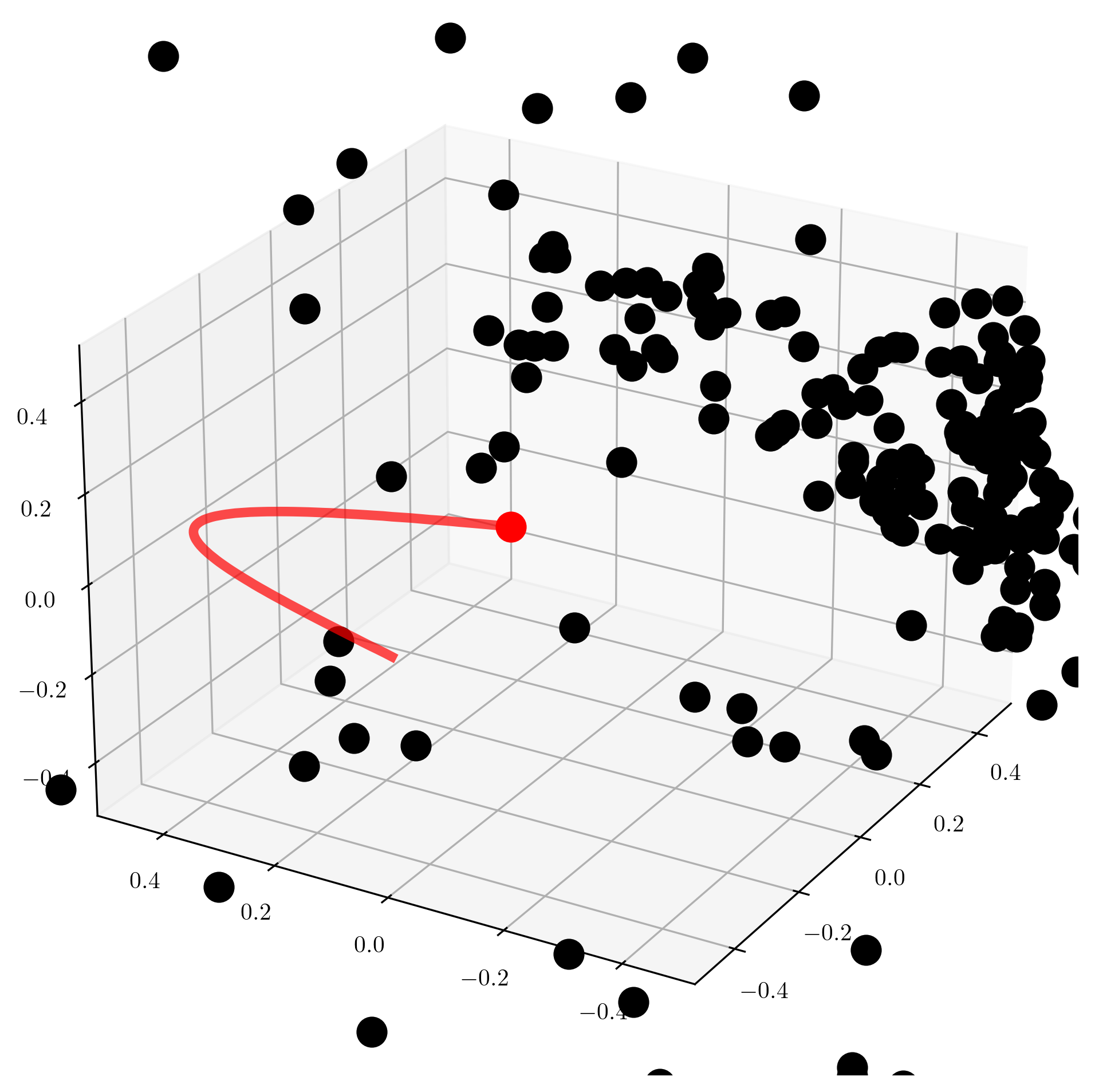}
    \\
    \includegraphics[width=0.22\textwidth]{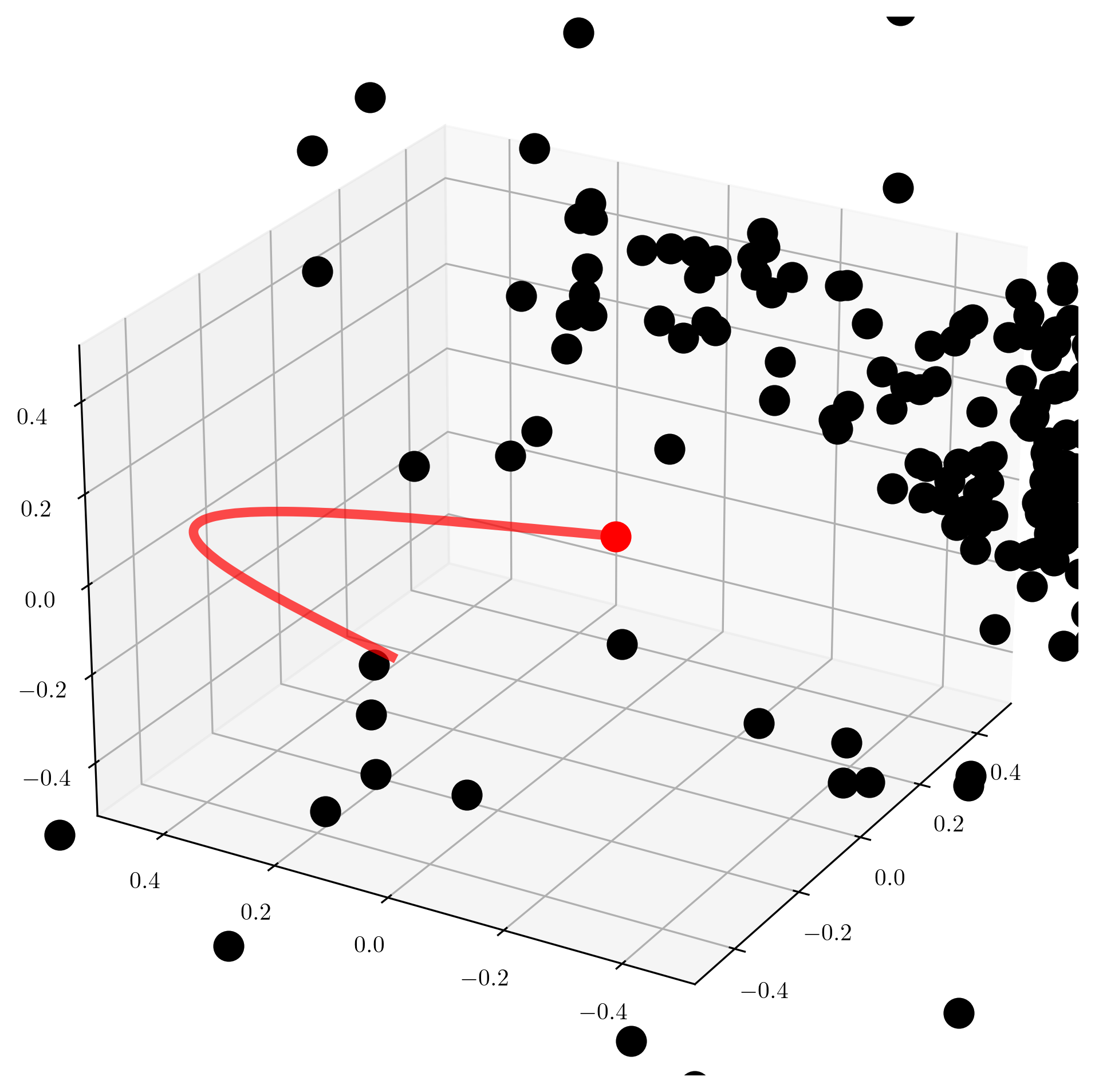}
    &
    \includegraphics[width=0.22\textwidth]{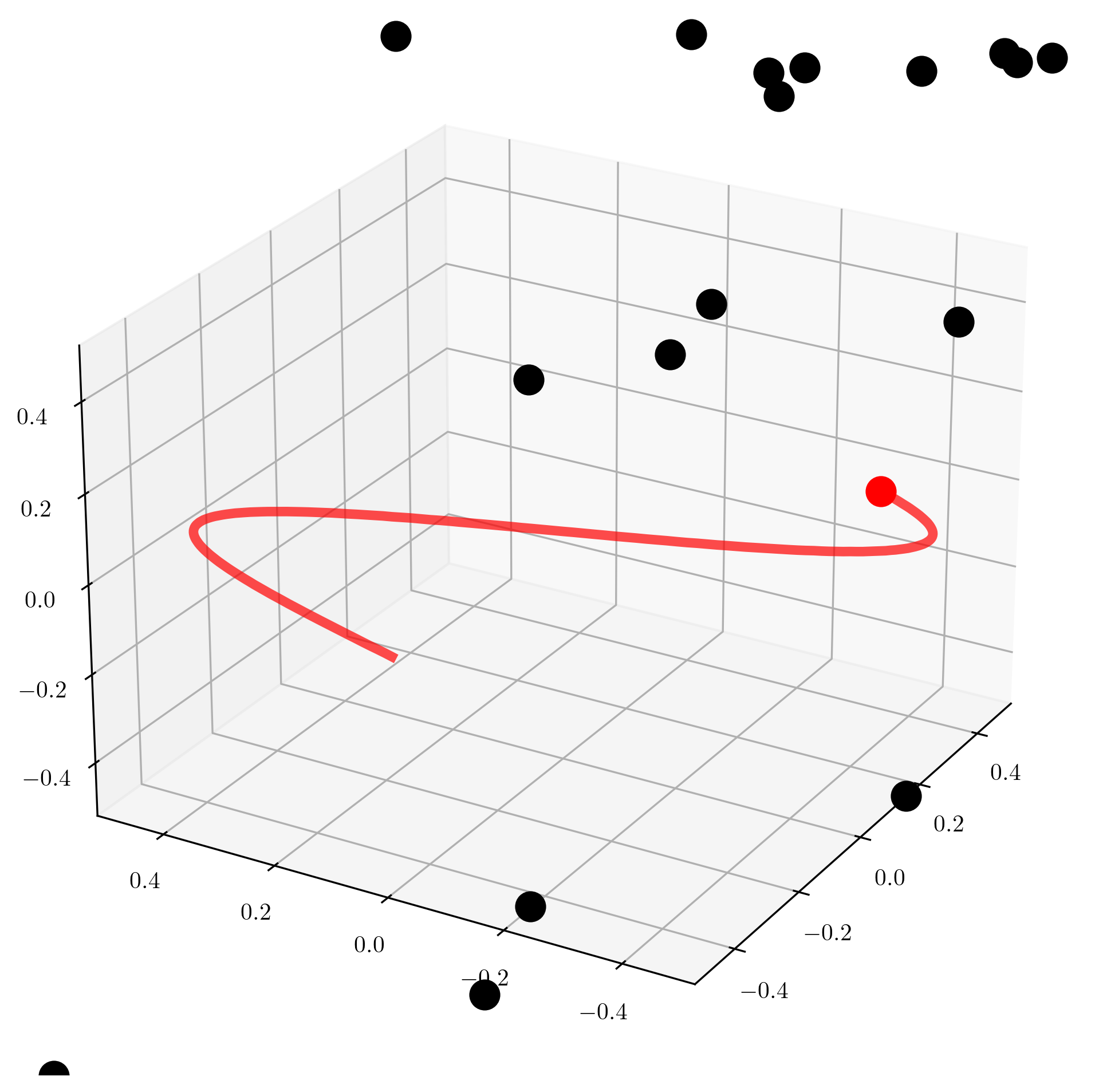}
    &
    \includegraphics[width=0.22\textwidth]{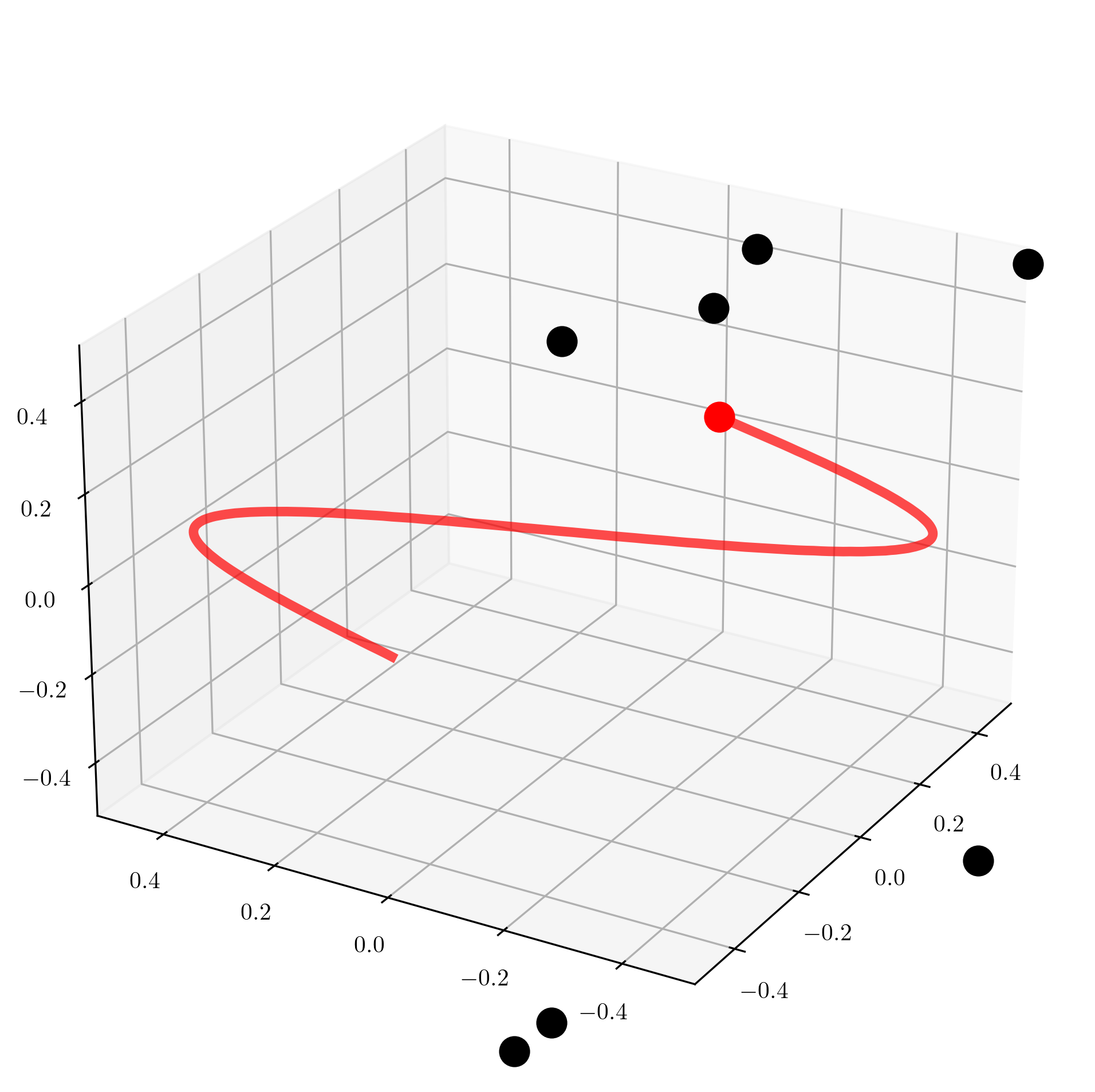}
    \end{tabular}
    \caption{Illustration of Swarm avoidance with the acceleration model. Swarm of agents (black dots) avoid  the red agent. The agent swarm must maintain an MMD value greater than $2\epsilon$ while moving the least amount possible as described in~\eqref{prob: CBF_avoidance}.}
    \label{fig: swarm_avoidance_3d}
\end{figure*}

For control-affine systems we have that
\begin{equation}\label{eq:K_CBF_lin}
\begin{split}
    &\mathcal{K}_{\text{CBF}}(s,\rho)\\
    =&\bigg\{ q~:~\int_{\mathbb{R}^d} \nabla  \delta_\rho \mathcal{H}(s,\rho) \cdot (A(s,x)+B(s,x)q(x)) \rho(x)dx \\
    &\geq -\partial_s \mathcal{H}(s,\rho) -\alpha\left(\mathcal{H}(s,\rho)\right) \bigg\}.
\end{split}
\end{equation}

Two critical remarks are in order.
\begin{enumerate}
    \item \eqref{eq: MF_CBF} is a quadratic program when $U$ is a convex polytope and the system is control-affine.
    \item In the mean-field setting, we only have \emph{one constraint} (excluding the a priori constraints $q(x)\in U$) as opposed to $n^2$ in the direct approach~\eqref{eq: general_CBF_swarm}. 
    
    % The number of inequality constraints (excluding the a priori constraints, $u(x)\in U$, on the controls) is 1 as opposed to $n^2$ in the direct approach~\eqref{eq: general_CBF_swarm}.
\end{enumerate}

These two points yield MF-CBF an \textit{efficient framework} for safe swarm-control.

\begin{figure*}[tb]
    \centering
    \begin{tabular}{ccc}
    \multicolumn{3}{c}{Swarm Tracking Example: 3D Double Integrator}
    \\
    \includegraphics[width=0.22\textwidth]{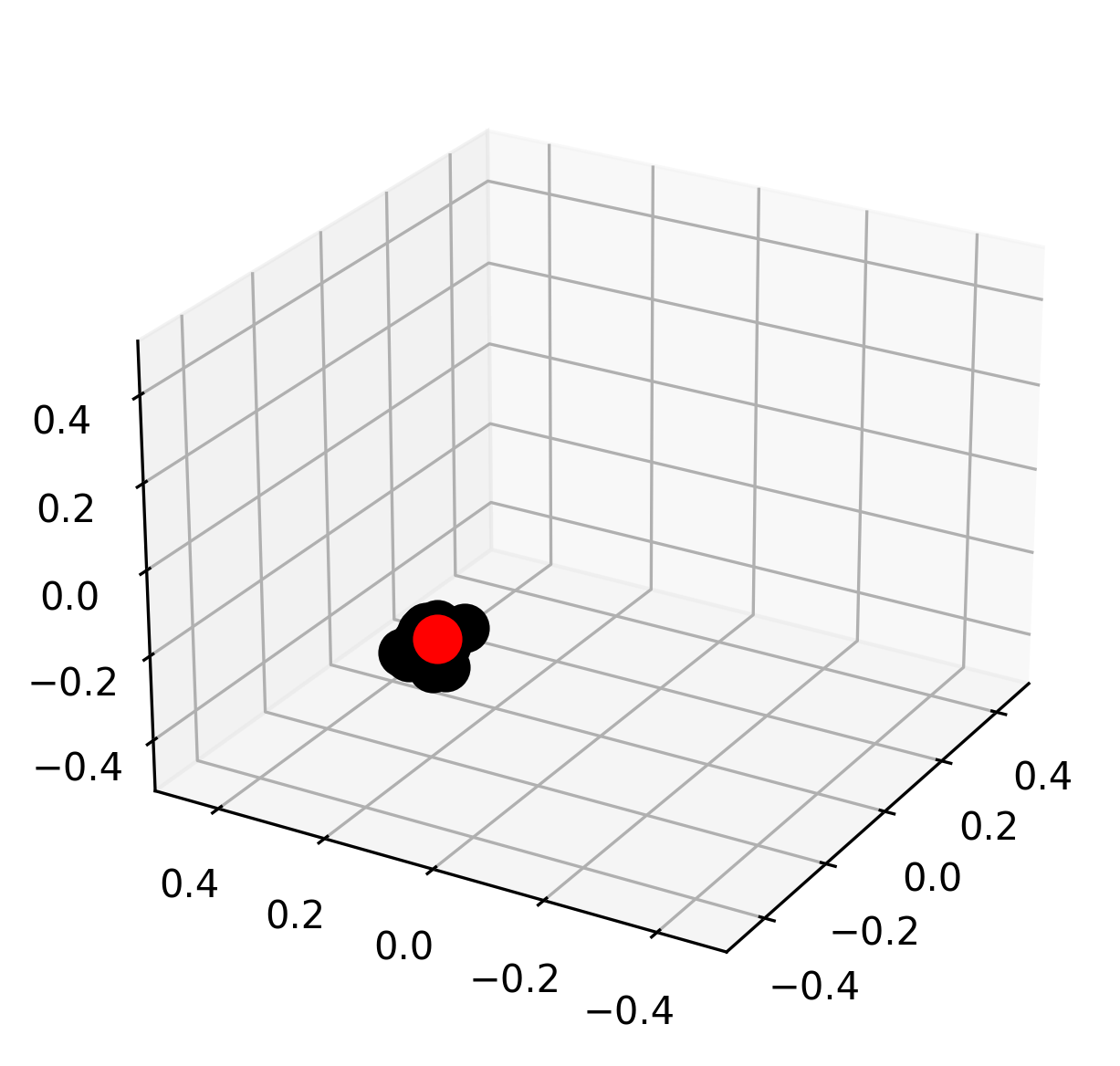}
    &
    \includegraphics[width=0.22\textwidth]{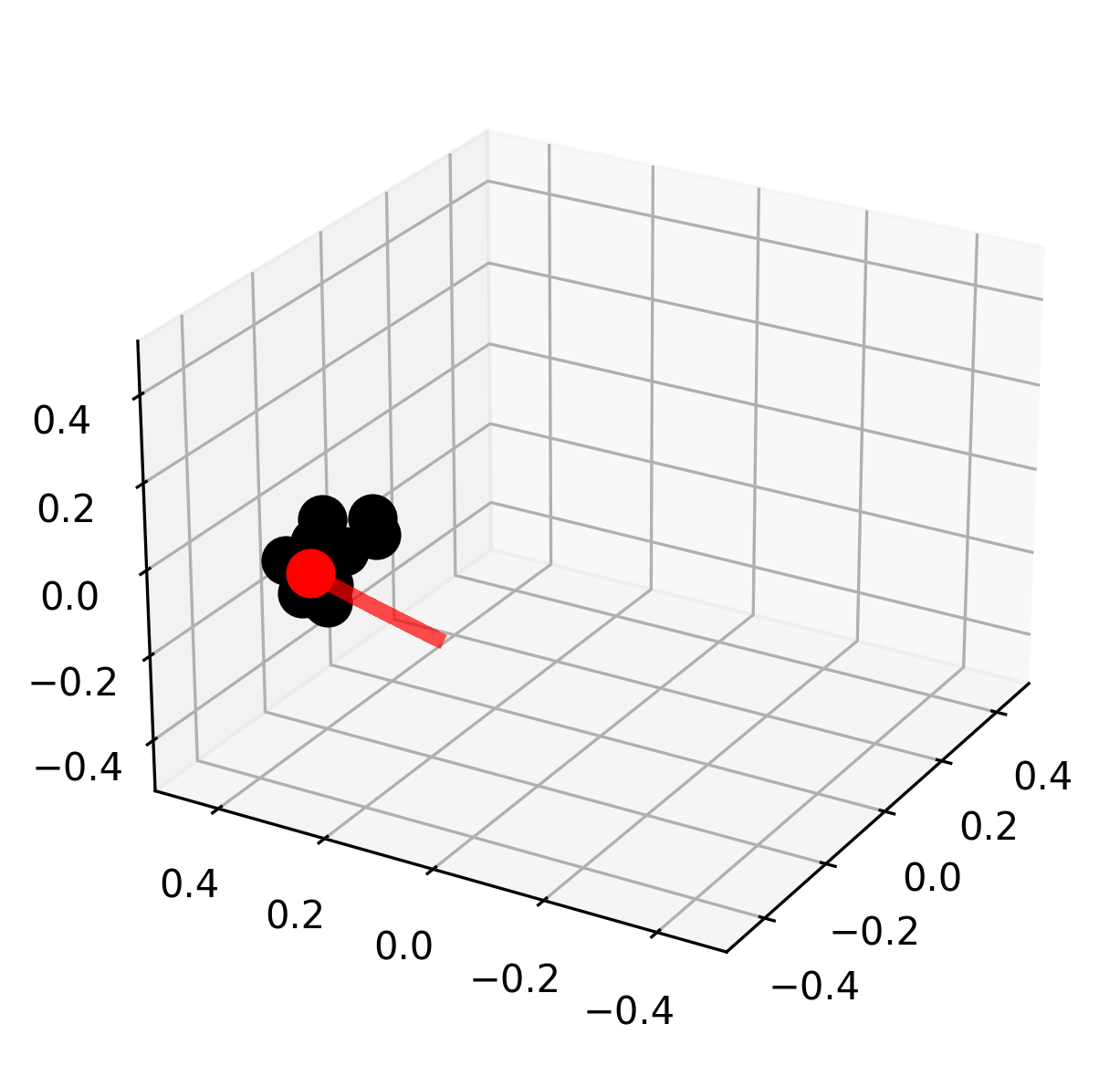}
    &
    \includegraphics[width=0.22\textwidth]{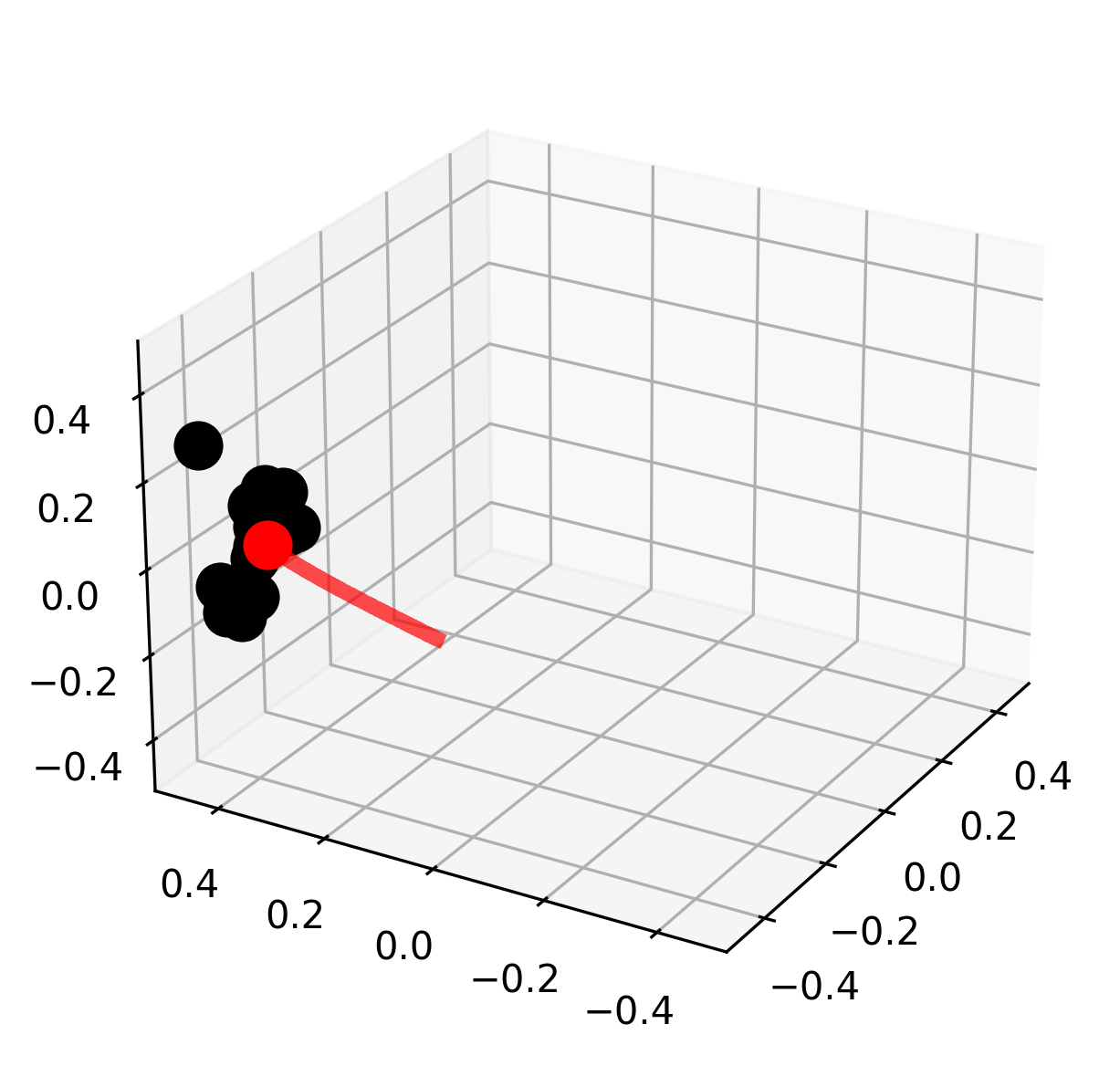}
    \\
    \includegraphics[width=0.22\textwidth]{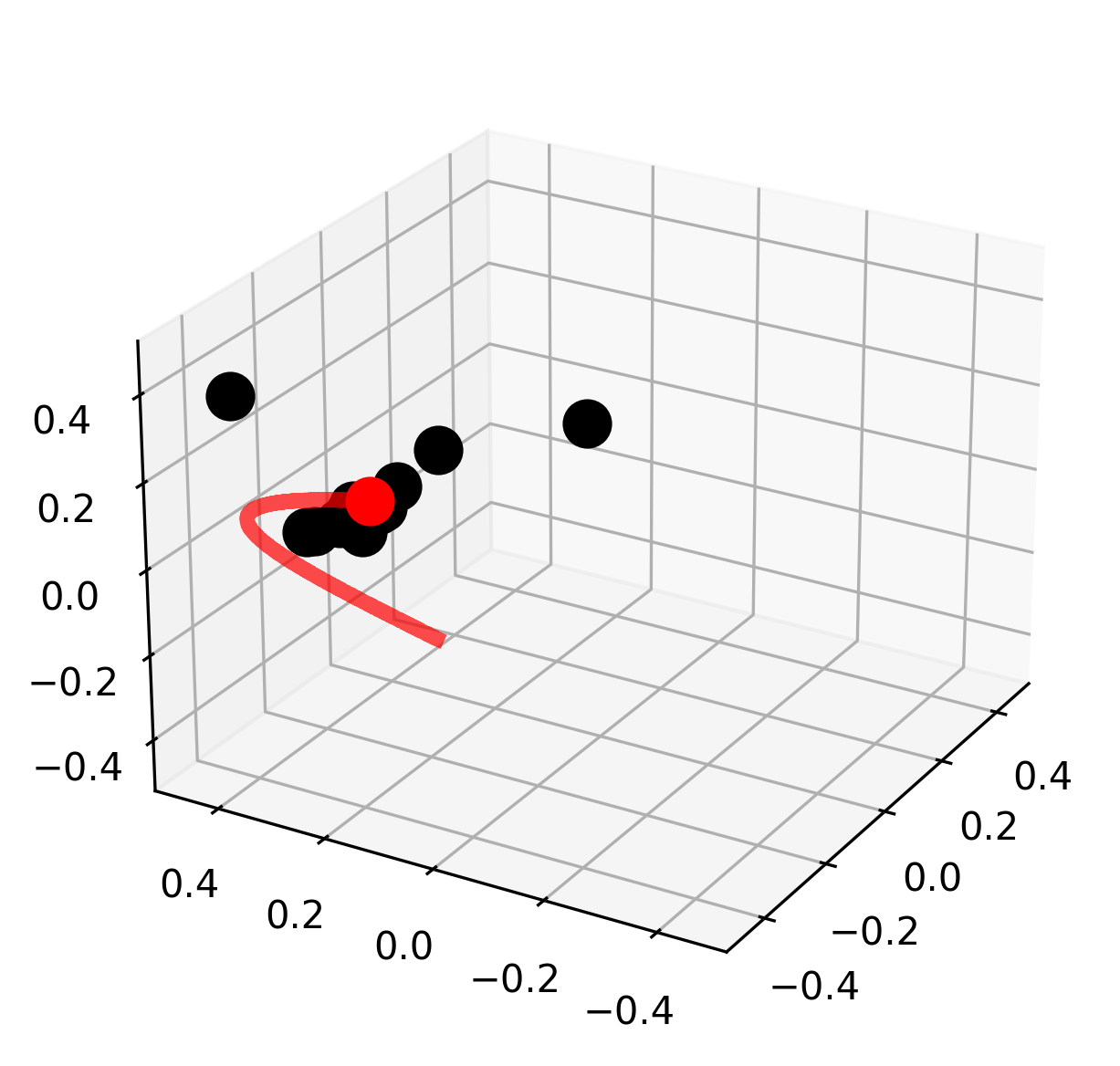}
    &
    \includegraphics[width=0.22\textwidth]{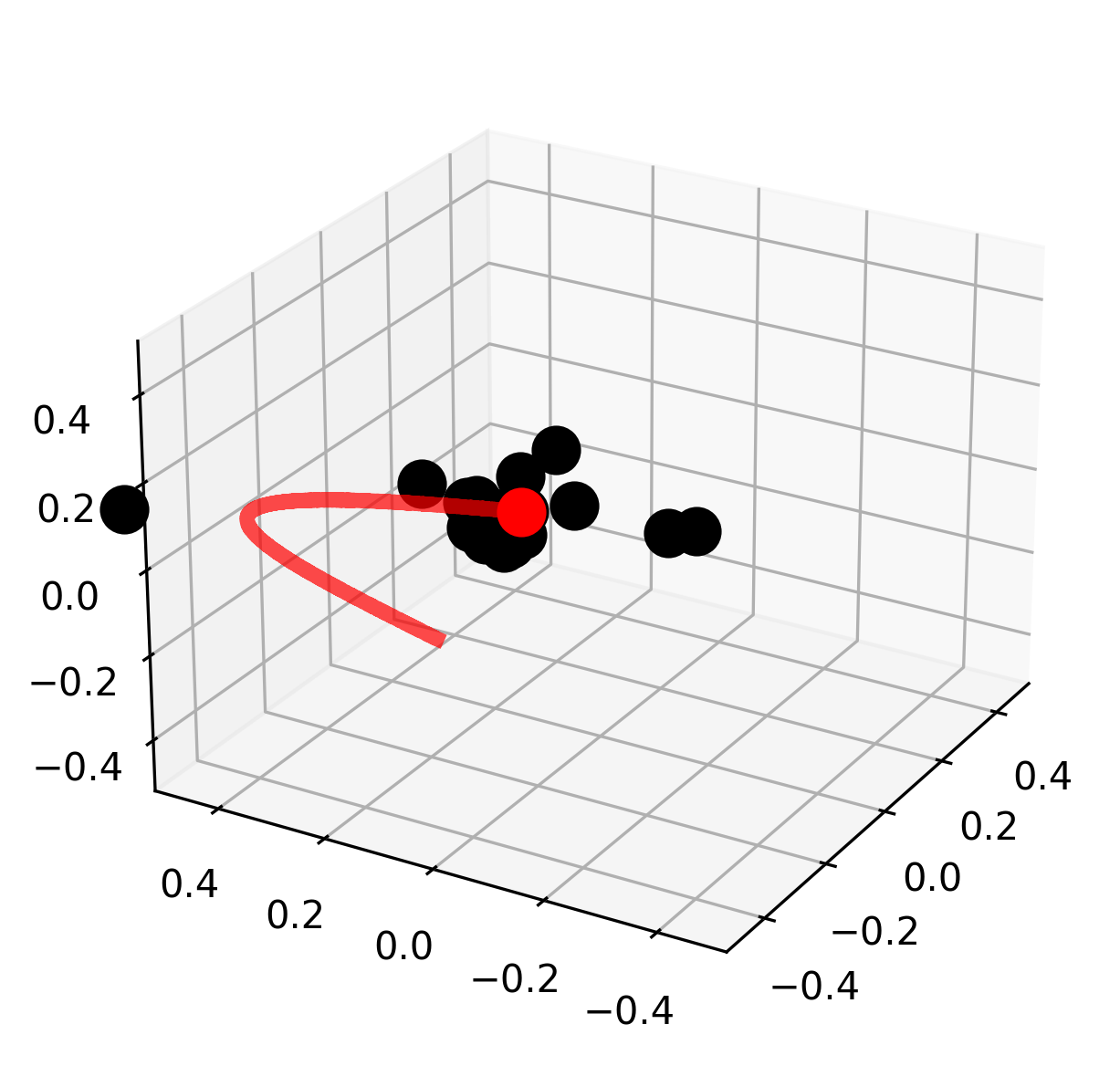}
    &
    \includegraphics[width=0.22\textwidth]{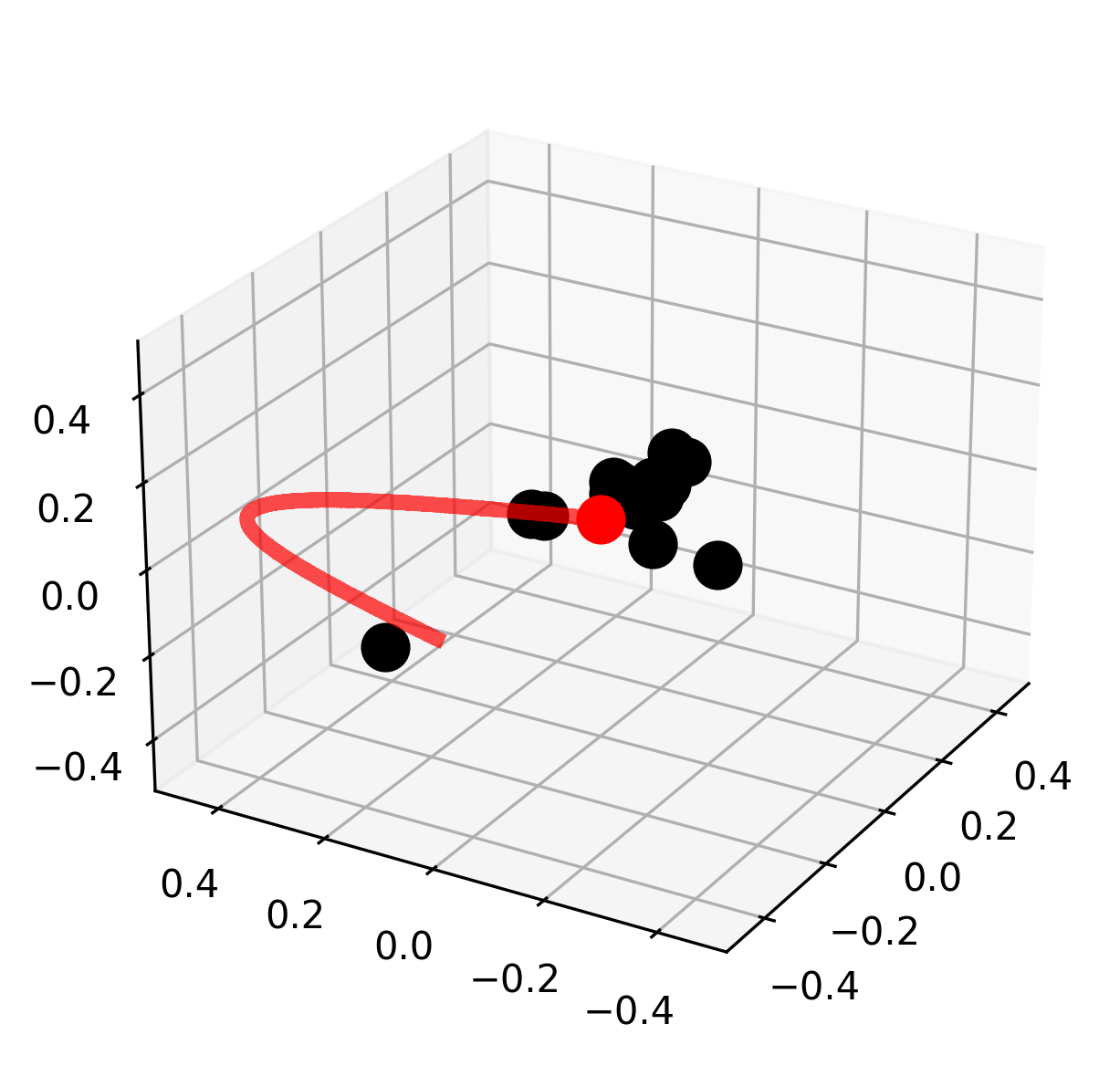}
    \end{tabular}
    \caption{Illustration of Swarm tracking with the acceleration model. Swarm of agents (black dots) and adversary (red dot). The agent swarm must maintain an MMD value \emph{less than} $2\epsilon$ while moving the least amount possible as shown in~\eqref{prob: CBF_tracking}.}
    \label{fig: swarm_tracking_3d}
\end{figure*}

\subsection{Examples}

Here we discuss applications of the MF-CBF framework in swarm avoidance and tracking examples.

\subsubsection{Swarm avoidance}

Suppose we wish to avoid (and maintain a certain distance from) an incoming object, denoted by $\rho^\dagger$. Mathematically, this condition can be formulated as
\begin{equation*}
    d(\rho(s,\cdot),\rho^\dagger(s,\cdot))\geq \epsilon, \quad \forall s \geq 0,
\end{equation*}
where $d$ is some distance function, and $\epsilon>0$. Although there are many choices for $d$, we consider the square maximum mean discrepancy (MMD) distance due to its analytic and computational simplicity\footnote{We note that other metrics such as Wasserstein distance may be considered.}. Hence, for a suitable choice of a symmetric positive-definite kernel $K$, we consider
\begin{equation}\label{eq:calH}
\begin{split}
    &\mathcal{H}(s,\rho)\\
    =&\frac{\text{MMD}(\rho,\rho^\dagger(s,\cdot))^2}{2}-\epsilon\\
    =&\frac{1}{2}\int_{\mathbb{R}^{2d}} K(x,y)(\rho(x)-\rho^\dagger(s,x))(\rho(y)-\rho^\dagger(s,y))dxdy-\epsilon.
\end{split}
\end{equation}
Next, we have that
\begin{equation}\label{eq:MMD_Frechet}
\begin{cases}
    \delta_\rho \frac{\text{MMD}^2(\rho,\rho^\dagger)}{2}=\int_{\mathbb{R}^d} K(x,y)(\rho(y)-\rho^\dagger(y))dy,\\
    \delta_{\rho^\dagger} \frac{\text{MMD}^2(\rho,\rho^\dagger)}{2}=\int_{\mathbb{R}^d} K(x,y)(\rho^\dagger(y)-\rho(y))dy.
\end{cases}
\end{equation}
Hence, we have that
\begin{equation}\label{eq:delta_rho_calH}
\begin{split}
    \delta_\rho \mathcal{H}(s,\rho)=\int_{\mathbb{R}^d} K(x,y)(\rho(y)-\rho^\dagger(s,y))dy,
\end{split}
\end{equation}
and
\begin{equation}\label{eq:partial_s_calH1}
\begin{split}
    \partial_s \mathcal{H}(s,\rho)=&\int_{\mathbb{R}^d} \delta_{\rho^\dagger} \left(\frac{\text{MMD}(\rho,\rho^\dagger(s,\cdot))^2}{2}-\epsilon \right) \partial_s \rho^\dagger(s,x)dx\\
    =&\int_{\mathbb{R}^{2d}} K(x,y)(\rho^\dagger(s,y)-\rho(y)) \partial_s \rho^\dagger(s,x)dxdy.
\end{split}
\end{equation}
Assuming $\rho^\dagger$ evolves according to the dynamics
\begin{equation}\label{eq:cont_eq_dagger}
    \partial_s \rho^\dagger(s,x)+\nabla \cdot (\rho^\dagger(s,x)v^\dagger(s,x))=0,
\end{equation}
we obtain that
\begin{equation}\label{eq:partial_s_calH}
\begin{split}
    &\partial_s \mathcal{H}(s,\rho)\\
    =&\int_{\mathbb{R}^{2d}} \nabla_x K(x,y)\cdot v^\dagger (s,x)(\rho^\dagger(s,y)-\rho(y)) \rho^\dagger(s,x)dxdy.
\end{split}
\end{equation}
Combining the derivations above, we find that
\begin{equation}\label{eq:K_CBF_MMD}
    \begin{split}
        &\mathcal{K}_{\text{CBF}}(s,\rho)= \bigg\{q:\\
        &\int_{\mathbb{R}^{2d}} \nabla_x K(x,y)\cdot f(s,x,q(s,x))(\rho(y)-\rho^\dagger(s,y))\rho(x)dxdy\\
        &+\int_{\mathbb{R}^{2d}} \nabla_x K(x,y)\cdot v^\dagger (s,x)(\rho^\dagger(s,y)-\rho(y)) \rho^\dagger(s,x)dxdy\\
        &+\alpha\left(\frac{\text{MMD}(\rho,\rho^\dagger(s,\cdot))^2}{2}-\epsilon\right) \geq 0 \bigg\}.
    \end{split}
\end{equation}

\begin{figure}[tb]
    \centering
    \begin{tabular}{cc}
    \textbf{a)} Swarm Avoidance
    &
    \textbf{b)} Swarm Tracking
    \\
    \includegraphics[width=0.22\textwidth]{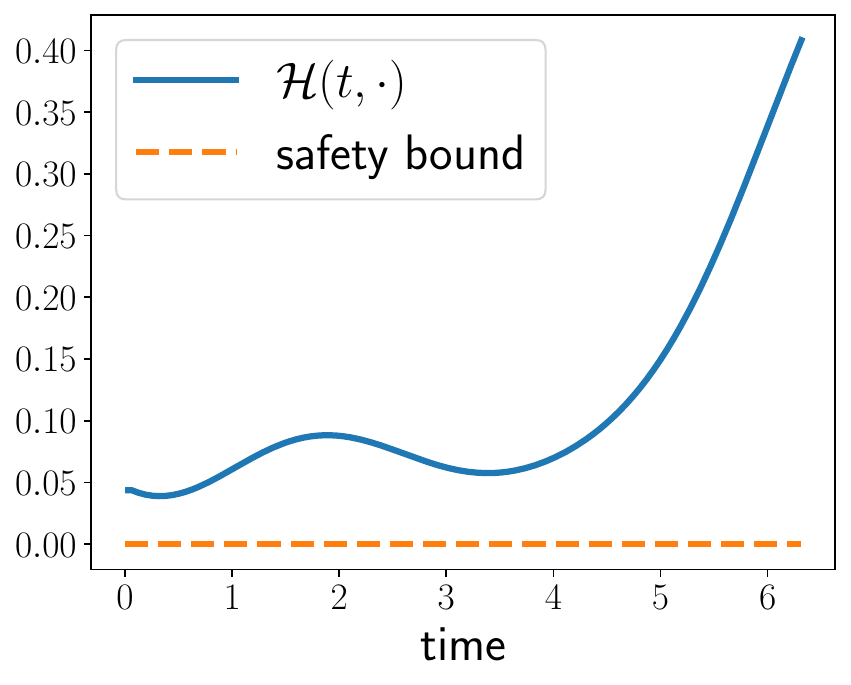}
    &
    \includegraphics[width=0.22\textwidth]{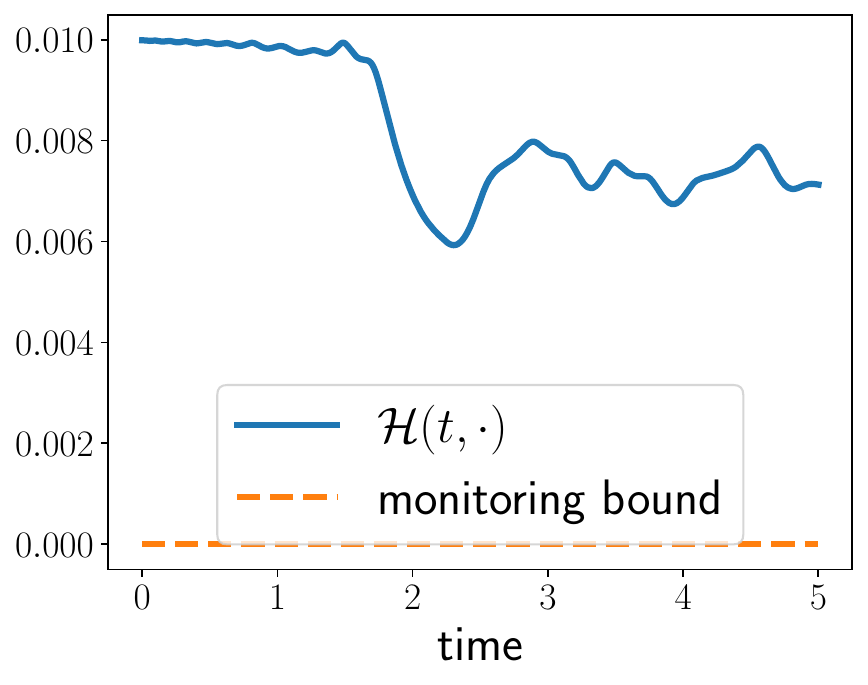}
    \end{tabular}
    \caption{Values of $\mathcal{H}$ over time (blue line) and monitoring bound that guarantees $\mathcal{H}(t,\rho) \geq 0$ (orange dashed line). In both examples, $\mathcal{H}(t,\rho) \geq 0$, which guarantees (a) $\frac{1}{2}\text{MMD}^2(\rho, \rho^\dagger) \geq \epsilon$ (safety of the swarm of agents from the $\rho^\dagger$ in red), and (b) $\frac{1}{2}\text{MMD}^2(\rho, \rho^\dagger) \leq \epsilon$ (proximity to $\rho^\dagger$ in red).}
    \label{fig: hval_history}
\end{figure}

\subsubsection{Swarm tracking}

The swarm avoidance framework in the previous section can be easily modified to a swarm tracking one by changing the sign of $\mathcal{H}$ in~\eqref{eq:calH}. Indeed, consider
\begin{equation}\label{eq:calH_tracking}
    \begin{split}
    &\mathcal{H}(s,\rho)\\
    =&\epsilon-\frac{\text{MMD}(\rho,\rho^\dagger(s,\cdot))^2}{2}\\
    =&\epsilon-\frac{1}{2}\int_{\mathbb{R}^{2d}} K(x,y)(\rho(x)-\rho^\dagger(s,x))(\rho(y)-\rho^\dagger(s,y))dxdy,
\end{split}
\end{equation}
where $\rho^\dagger$ is now the distribution the swarm that one wants to track. Then~\eqref{eq:calH>=0} would ensure that $\rho$ stays close to $\rho^\dagger$ all the time. Recycling the calculations for swarm avoidance yields
\begin{equation}\label{eq:K_CBF_MMD_tracking}
    \begin{split}
        &\mathcal{K}_{\text{CBF}}(s,\rho)= \bigg\{q:\\
        &\int_{\mathbb{R}^{2d}} \nabla_x K(x,y)\cdot f(s,x,q(s,x))(\rho(y)-\rho^\dagger(s,y))\rho(x)dxdy\\
        &+\int_{\mathbb{R}^{2d}} \nabla_x K(x,y)\cdot v^\dagger (s,x)(\rho^\dagger(s,y)-\rho(y)) \rho^\dagger(s,x)dxdy\\
        &-\alpha\left(\epsilon-\frac{\text{MMD}(\rho,\rho^\dagger(s,\cdot))^2}{2}\right)\leq 0 \bigg\}.
    \end{split}
\end{equation}

\section{Experiments}
\label{sec: experiments}

We illustrate the effectiveness of the MF-CBF framework on two types of applications: swarm avoidance and swarm tracking. 
We note that in practice we typically have access to \emph{samples} of $\rho$ in~\eqref{eq: MF_CBF}; consequently, we have a discrete approximation where $q$ and $q^\star$ are vectors and the objective is given by the Euclidean $l_2$ norm instead.
The dynamics used in both applications are double integrator dynamics, that is, 
$f(t,z,u) = Az + Bu$ where $z$ stacks the position and velocity of the agents. Our experiments are coded in python; in particular, the quadratic programs arising from the MF-CBFs are solved with cvxpy~\cite{diamond2016cvxpy}, an open source library for solving convex optimization problems. 
In both setups, the nominal controller is given by $q^* = 0$, i.e., we would like the swarm to have constant velocity (or remain still if they are already stationary).

\begin{comment}
We model both the controlled swarm and adversaries by empirical distributions; that is,
\begin{equation}\label{eq:empirical}
    \rho(s,\cdot)= \frac{1}{n} \sum_{i=1}^n \delta_{z_i(s)},\quad \rho^\dagger(s,\cdot)= \frac{1}{n} \sum_{j=1}^m \delta_{z^\dagger_j(s)}.
\end{equation}
Hence, we can model distributed controls as
\begin{equation}\label{eq:empirical_controls}
\begin{split}
    &(q_1(s),q_2(s),\cdots,q_n(s))\\
    =&(q(s,z_1(s)),q(s,z_2(s)),\cdots,q(s,z_n(s)))
\end{split}
\end{equation}    
\end{comment}

\subsection{Swarm Avoidance}
In these experiments, we suppose we have a swarm of agents that are stationary, and as soon as a moving object gets too close, the swarm of agents avoid the object; this is akin to a swarm of fish avoiding an incoming shark. 
The MF-CBF problem to be solved at each time step is then given by 
\begin{equation}
    \min_q \|q\|_2 \quad \text{s.t.} \quad q \in \mathcal{K}_{\text{CBF}} \quad \text{ in } \quad \eqref{eq:K_CBF_MMD} 
    \label{prob: CBF_avoidance}
\end{equation}

\begin{comment}
\begin{equation}
    \min_u \|u\|_2 \quad \text{s.t.} \quad v \in \mathcal{K}_{\text{CBF}} \quad \text{ in } \quad \eqref{eq:K_CBF_MMD} 
    \label{prob: CBF_avoidance}
\end{equation}    
\end{comment}

Figure~\ref{fig: swarm_avoidance_3d} shows the effectiveness of the MF-CBF in the swarm avoidance setting.
In this experiment, the red dot is a moving obstacle has a constant velocity. 
Guaranteed safety is shown in Fig.~\ref{fig: hval_history}, where $\mathcal{H}(\rho(s, \cdot)) \geq 0$ for all $s$, which guarantees that $\frac{1}{2}\text{MMD}^2(\rho, \rho^\dagger) \geq \epsilon$.

\subsection{Swarm Tracking}
In this experiments, we suppose we have a swarm of agents that must maintain a certain distance from the red agent. In particular, they must maintain an MMD squared value \emph{less than} $2\epsilon$. 
The MF-CBF problem to be solved at each time step is then given by 
\begin{equation}
    \min_q \|q\|_2 \quad \text{s.t.} \quad q \in \mathcal{K}_{\text{CBF}} \quad \text{ in } \quad \eqref{eq:K_CBF_MMD_tracking} 
    \label{prob: CBF_tracking}
\end{equation}

\begin{comment}
\begin{equation}
    \min_u \|u\|_2 \quad \text{s.t.} \quad v \in \mathcal{K}_{\text{CBF}} \quad \text{ in } \quad \eqref{eq:K_CBF_MMD_tracking} 
    \label{prob: CBF_tracking}
\end{equation}    
\end{comment}

As before, Figure~\ref{fig: swarm_tracking_3d} shows the effectiveness of the mean-field CBF in the swarm tracking. In this experiment, the red agent has a constant velocity.  
Feasibility of the tracking distance is shown in Fig.~\ref{fig: hval_history}, where $\mathcal{H}(\rho(s, \cdot)) \geq 0$ for all $s$, which guarantees that $\frac{1}{2}\text{MMD}^2(\rho, \rho^\dagger) \leq \epsilon$.

\subsection{Discussion}
Just like traditional CBFs, there are two major considerations when employing MF-CBFs: the choice of $\alpha(x)$ and the necessary time discretization, both of which are nuanced tasks and context-dependent~\cite{ames2019control}. In our experiments, these were hyperparameters that we tuned until we obtained the desired performance; for instance, the swarm tracking problems required a much finer time-discretization since the projection onto the set of controls in~\eqref{eq:MFCBF_inclusion} was activated more frequently. 
Finally, we remark that since we have a finite number of agents, the distributions are comprised of Dirac-delta functions so that the norm $\|\cdot\|_2$ in~\eqref{prob: CBF_avoidance} and~\eqref{prob: CBF_tracking} are equivalent to $\| \cdot \|_{L^2(\rho(s,\cdot))}$ in~\eqref{eq: MF_CBF}.
Code details and accompanying videos can be found in \url{https://github.com/mines-opt-ml/mean-field-cbf}.

\section{Conclusion}
We present MF-CBF, a mean-field framework for real-time swarm control. The core idea is to extend CBFs to the space of distributions. Our numerical experiments show MF-CBFs are effective in a swarm avoidance and a swarm tracking example. Future work involves employing these in optimal control settings where the feedback control is available~\cite{onken2020neural,onken2022neural, onken2021ot, vidal2023taming} and improving their computational efficiency via kernel decoupling techniques~\cite{nurbekyan2019fourier, chow2022numerical, agrawal2022random, liu2021computational, vidal2024kernel}.

% \section*{Acknowledgment}

% \section*{References}

% \newpage
\bibliographystyle{ieeetr}
\bibliography{refs}

\end{document}